\documentclass[11pt,oneside]{amsart}
\usepackage{fouriernc} 

\usepackage{Environments}   
\usepackage{Definitions}    
\usepackage{PageSetup}      

\usepackage{multicol}

\renewcommand{\theenumi}{\textit{\alph{enumi}}}

\renewcommand{\theenumii}{\textit{\roman{enumii}}}

\makeatletter
\newcommand{\xRightarrow}[2][]{\ext@arrow 0359\Rightarrowfill@{#1}{#2}}
\makeatother

\theoremstyle{plain}
\newtheorem*{problem}{Problem}

\hypersetup{
 pdfkeywords={lax maps,strict maps,2-monad},
 pdfauthor={Nick Gurski,Niles Johnson,Angelica Osorno},
}


\usepackage{AuthorInfo}     

\title{Extending homotopy theories across adjunctions}
\authorGurski
\authorJohnson
\authorOsorno
\date{2016-12-07}

\begin{document}

\begin{abstract}
  Constructions of spectra from symmetric monoidal categories 
  are typically functorial with respect to strict structure-preserving
  maps, but often the maps of interest are merely lax monoidal.  
  We describe conditions under which one can transport
  the weak equivalences from one category to another with the same objects and
  a broader class of maps. Under mild hypotheses this process produces an 
  equivalence of homotopy theories.  We describe examples including algebras over an
  operad, such as symmetric monoidal categories and $n$-fold monoidal
  categories; and diagram categories, such as $\Gamma$-categories.  
\end{abstract}

\maketitle

\section*{Introduction}

The classifying space functor from categories to topological spaces
provides a way of constructing spaces with certain algebraic
structure. Of particular importance are infinite loop space machines,
which construct spectra out of structured categories such as 
symmetric monoidal categories
\cite{Sta1971Hspaces,Qui1973Higher,May1974Einfty,Seg74Categories,WaldKtheory,EM2006Rings,May09construction,Oso10Spectra}. 
The discussion of the functoriality of these constructions is somewhat
nuanced
due to the range of possible morphisms one might choose.  These 
morphisms 
differ in \emph{strength}, the degree to which the
underlying functors of structured categories preserve the
structure. 

It is often the case that such machines are
obviously functorial with respect to maps that strictly preserve the ambient
structure. This is the case, for example, for the operadic machine and
maps of symmetric monoidal categories. The maps that arise in
practice however---for example, the functors of
module categories induced by a morphism of commutative rings---are
typically not strict, but strong or merely lax, meaning that they
preserve monoidal structure up to coherent isomorphism or merely
coherent morphism.  One way to 
handle such variation is to
construct variant machinery for each type of morphism and prove that 
the corresponding
constructions are equivalent. This allows one to prove general
theorems about the strict case, for example, Segal machinery and
strict maps of $\Ga$-categories, but apply them to the more broadly useful
strong or lax case, for example, Segal machinery and lax
maps.  Such an approach appears in a number of places in the
literature, for example, in \cite[\S3]{Man10Inverse}. 

In this paper we consider a more systematic approach: a direct
comparison of the homotopy theories arising from structured categories
and maps of various strength.
For this purpose, 
we discuss homotopy theory in the generality of relative categories.
A relative category is merely a category $\cC$ equipped with a
subcategory $\cW$ containing all of the objects.  The morphisms in 
this
subcategory then play the role of weak equivalences.  A pair $(\cC, 
\cW)$
presents a homotopy theory \cite{Rez01Model,BK12Characterization}, 
and such a presentation neatly hides, but crucially still retains, 
higher homotopical information, such as mapping spaces, that is not 
present in the bare homotopy category.  

The central problem we address in this paper may then be 
described as follows. 
\begin{problem}
For a homotopy theory $(\cC, \cW)$, give 
criteria for enlarging the class of morphisms in $\cC$ to give a 
new category $\cC'$ with a larger class of weak 
equivalences $\cW'$ such that the inclusion 
$(\cC, \cW) \hookrightarrow (\cC', \cW')$ is an equivalence of 
homotopy theories.  
\end{problem}
\noindent In other words, how can we replace the morphisms in $\cC$ with more 
flexible ones without changing the homotopy theory?  The advantages 
of such a strategy are well-known: the smaller class of morphisms 
is likely more amenable to abstract manipulation, while the larger 
class will often arise in examples of interest.

Our first main result, \cref{thm:equiv-conditions-left-merged}, gives
general conditions under which one can extend the class of weak
equivalences via an adjunction
\[
  \begin{xy}
    0;<15mm,0mm>:<0mm,10mm>:: 
    (-1,0)*+{\cC_\tau}="ct";
    (1,0)*+{\cC_\la,}="cl";
    {\ar@/_3mm/_-{i} "ct"; "cl"};
    {\ar@/_3mm/_-{Q} "cl"; "ct"};
    (0,0)*+{\bot}="v";
  \end{xy}
\]
with $i$ the identity on objects.  Moreover, we prove that this extension is
unique and that the resulting homotopy theories are equivalent.  This
should be seen as the relative-categorical analogue of a very strong
kind of transferred model structure for Quillen model categories (see, for example, \cite{CransQuillen}).

Our second main result, \cref{thm:algebra-rel-var-adj-equiv}, takes
$\cC_\tau$, respectively $\cC_\la$, to be categories of algebras and
strict, respectively lax, maps for a 2-monad $T$ on a 2-category
$\cK$.  In particular, $T$ may be the 2-monad on $\cK = \Cat$ whose
algebras are symmetric monoidal categories.  This special case
provides an enhancement of previous work in the case of symmetric
monoidal categories: Thomason and Mandell show that the corresponding
homotopy categories are equivalent after localizing stable
equivalences \cite[Lemma 1.9.2]{Tho95Symmetric} and weak equivalences
\cite[Theorem 3.9]{Man10Inverse}, respectively.  Many other examples
of interest arise in this way, and we describe a number of them in
detail.

We choose the framework of relative
categories, rather than Quillen model structures, as many of the
categories we encounter are not well-behaved enough to construct 
model
structures.  For example, one variant of our results (see 
\cref{thm:sym-mon-cat-st-eq}) shows that the 
homotopy theory
of symmetric monoidal categories using strict symmetric monoidal
functors and stable equivalences extends uniquely to an equivalent homotopy theory on the category 
of symmetric
monoidal categories using \emph{lax} symmetric monoidal functors. 
While it is straightforward to define compatible weak
equivalences in these categories, 
the
latter category is neither complete nor cocomplete so, in particular,
constructing a model structure via the small object argument is not
possible.

\subsection*{Outline}
In \cref{sec:adjunctions-creating-weak-equivalences} we recall basic
notions of relative categories and give our first main result
regarding equivalences of homotopy theories for strict and lax maps.

In \cref{sec:applications} we apply the results of
\cref{sec:adjunctions-creating-weak-equivalences} to the different
morphism variants for algebras over a 2-monad using the factorization
system techniques of Bourke and Garner \cite{BG2014AWFSII}.
We then go on to give
the following examples: symmetric monoidal categories and $n$-fold
monoidal categories (\cref{sec:K-is-Cat}); categories with group
actions (\cref{sec:diagrams}); and $\Ga$-categories or
$\Ga$-2-categories (\cref{sec:reduced-diagrams}).  In each case we
discuss interesting map variants, classes of weak equivalences, and
explicitly state the resulting equivalence of homotopy theories.

In \cref{sec:2D-monadicity} we recall Bourke's theory of 2-dimensional
monadicity \cite{Bou2014Two}.  We use this theory to recognize some
naturally-occurring morphisms as the lax algebra morphisms for various
2-monads, thus completing the proofs required for some of the examples
in \cref{sec:applications}.

\addtocontents{toc}{\SkipTocEntry}
\subsection*{Acknowledgments} 

The first author was supported by EPSRC EP/K007343/1. The third author
was partially supported by a grant from the Simons Foundation
(\#359449, Ang\'{e}lica Osorno).  The authors thank Mark Behrens for a
number of useful comments.

\section{Adjunctions creating weak equivalences}
\label{sec:adjunctions-creating-weak-equivalences}

In this section we develop the fundamental machinery to extend a
notion of weak equivalence from a given category to one with the same
set of objects and a larger class of morphisms.
We give conditions which guarantee that this extension yields an
equivalence of homotopy theories.

To begin, we recall the elementary notions of relative categories. For
more details, see
\cite{DK80Simplicial,Rez01Model,BK12Characterization}.

\begin{defn}
A \emph{relative category} is a pair $(\cC,\cW)$ in which $\cC$ is a
category and 
$\cW$ is a subcategory of $\cC$ containing all of the objects.  A
\emph{relative functor} $F \cn (\cC,\cW) \to (\cC',\cW')$ is a functor
$F \cn \cC \to \cC'$ such that $F$ restricts to a functor
$\cW \to \cW'$. A \emph{relative adjunction} is an adjunction
 \begin{equation}\label{eqn:reladj}
  \begin{xy}
    0;<15mm,0mm>:<0mm,10mm>:: 
    (-1,0)*+{(\cC, \cW)}="ct";
    (1,0)*+{(\cD, \cV),}="cl";
    {\ar@/_3mm/_-{U} "ct"; "cl"};
    {\ar@/_3mm/_-{F} "cl"; "ct"};
    (0,0)*+{\bot}="v";
  \end{xy}
  \end{equation}
  where $F$ and $U$ are relative functors.
\end{defn}

\begin{defn}
  A \emph{category with weak equivalences} is a relative category
  $(\cC,\cW),$ where $\cW$ contains all isomorphisms and satisfies the
  2-out-of-3 property.  We informally refer to a category with weak
  equivalences as a homotopy theory.
\end{defn}

\begin{defn}
  Let $(\cC,\cW)$ and $(\cD,\cV)$ be categories with weak
  equivalences.  We say a functor $F\cn\cC \to \cD$ \emph{creates weak
    equivalences} if for each morphism $f$ of $\cC$, $f \in \cW$ if
  and only if $Ff \in \cV$.
\end{defn}

We now recall the definition of equivalence between homotopy theories \cite{Rez01Model}.
This notion is equivalent to the requirement that the induced map on
hammock localizations be a DK-equivalence, and implies that the
induced map on categorical localizations is an equivalence \cite{BK12Characterization}.

\begin{defn}
  A relative functor $F \cn (\cC,\cW) \to (\cD,\cV)$ is an
  \emph{equivalence of homotopy theories} if, in the complete Segal
  space model structure, the induced map on fibrant replacements of
  classification diagrams is a weak equivalence.  
\end{defn}
  
\begin{convention}
  Given a collection of weak equivalences, $\cW$, and a natural
  transformation $\eta$, we say that $\eta$ is a weak equivalence and
  write $\eta \in \cW$ if each component of $\eta$ is in $\cW$.
\end{convention}

For reference, we record the following observation.  Further
discussion appears in \cite[2.9]{GJO2017KTheory}.

\begin{lem}\label{lem:ajd-equiv-is-equiv}
  A relative adjunction whose unit and counit are weak equivalences
  induces an equivalence of homotopy theories.
\end{lem}

This lemma motivates the following definition.
\begin{defn}
  We say that a relative adjunction is an \emph{adjoint equivalence of
    homotopy theories} if the components of its unit and counit are
  weak equivalences.
\end{defn}

We are interested in the interplay between different types of
morphisms between given objects, and thus make the following
definition.

\begin{defn}
  A \emph{map extension} of a category $\cC_\tau$ is an inclusion
  $i\cn\cC_\tau \monoto \cC_\la$ which is the identity on objects.  We
  refer to the morphisms of $\cC_\tau$ as \emph{tight}, and those of
  $\cC_\la$ as \emph{loose}.
\end{defn}

For example, one might take the tight maps between monoidal categories
to be the strict monoidal functors and the loose maps to be the lax
monoidal functors, the oplax monoidal functors, the strong monoidal
functors, etc.  A map extension is a special case of what
\cite{LS2012Enhanced} call an $\sF$-category.

\begin{defn}
  Let
  \[
  \begin{xy}
    0;<15mm,0mm>:<0mm,10mm>:: 
    (-1,0)*+{(\cC_\tau, \cW_\tau)}="ct";
    (1,0)*+{(\cC_\la, \cW_\la)}="cl";
    {\ar@/_3mm/_-{i} "ct"; "cl"};
    {\ar@/_3mm/_-{Q} "cl"; "ct"};
    (0,0)*+{\bot}="v";
  \end{xy}
  \]
  be a relative adjunction.  We say \emph{$Q\dashv i$ creates weak
    equivalences} if both $i$ and $Q$ create weak equivalences.
\end{defn}
It is generally not the case that a relative adjunction creates a
class of weak equivalences in this sense.  We will, however, describe
useful hypotheses which guarantee this in a number of interesting
examples.

\begin{thm}\label{thm:equiv-conditions-left-merged}
  Let $(\cC_\tau, \cW_\tau)$, $(\cC_\la, \cW_\la)$ be categories with weak
  equivalences and let 
  \[
  \cC_\tau \fto{i} \cC_\la
  \]
  be a map extension.  Assume there is a left adjoint $Q \dashv i$
  with counit $\epz$ and unit $\eta$.  Then the following are
  equivalent:
  \begin{samepage}
  \begin{enumerate}
  \item\label{ecl-i-epz} $i$ creates weak equivalences and $\epz \in \cW_\tau$, 
  \item\label{ecl-i-eta} $i$ creates weak equivalences and $\eta \in \cW_\la$,
  \item\label{ecl-Q-epz} $Q$ creates weak equivalences and $\epz \in \cW_\tau$.
  \end{enumerate}
  Moreover, these conditions imply the following:
  \begin{enumerate}
    \setcounter{enumi}{3}
  \item\label{ecl-Q-eta} $Q$ creates weak equivalences and $\eta \in \cW_\la$.
  \end{enumerate}
  Consequently, 
  \[
  \begin{xy}
    0;<15mm,0mm>:<0mm,15mm>:: 
    (-1,0)*+{(\cC_\tau, \cW_\tau)}="ct";
    (1,0)*+{(\cC_\la, \cW_\la)}="cl";
    {\ar@/_3mm/_-{i} "ct"; "cl"};
    {\ar@/_3mm/_-{Q} "cl"; "ct"};
    (0,0)*+{\bot}="v";
  \end{xy}
  \]
  is an adjoint equivalence of homotopy
  theories.
  \end{samepage}
  \end{thm}

\begin{proof}
  With $Q \dashv i$, we have the following triangle identities for $A
  \in \cC_\tau$ and $B \in \cC_\la$.
  \[
  \begin{xy}
    0;<19mm,0mm>:<0mm,14mm>:: 
    (0,0)*+{QB}="A";
    (1,0)*+{QiQB}="B";
    (1,-1)*+{QB}="C";
    {\ar^{Q\eta_B} "A"; "B"};
    {\ar^{\epz_{QB}} "B"; "C"};
    {\ar@{=} "A"; "C"};
  \end{xy}
  \hspace{1in}
  \begin{xy}
    0;<19mm,0mm>:<0mm,14mm>:: 
    (0,0)*+{iA}="A";
    (1,0)*+{iQiA}="B";
    (1,-1)*+{iA}="C";
    {\ar^{\eta_{iA}} "A"; "B"};
    {\ar^{i\epz_{A}} "B"; "C"};
    {\ar@{=} "A"; "C"};
  \end{xy}
  \]

  We first show \eqref{ecl-i-epz} $\Leftrightarrow$ \eqref{ecl-i-eta}
  and \eqref{ecl-Q-epz} $\Rightarrow$ \eqref{ecl-Q-eta}: If $i$
  creates weak equivalences and $\epz_A \in \cW_\tau$, then $i \epz_A \in
  \cW_\la$ and therefore $\eta_{iA} \in \cW_\la$ by the 2-out-of-3
  property.  But since $i$ is the identity on objects, we have that
  $\eta \in \cW_\la$.  Conversely, $\eta \in \cW_\la$ implies each $i
  \epz_A \in \cW_\la$ and therefore $\epz \in \cW_\tau$ since $i$ creates
  weak equivalences.  Likewise, if $Q$ creates weak equivalences and
  $\epz \in \cW_\tau$, then $\eta \in \cW_\la$.

  Now we show \eqref{ecl-i-epz} and \eqref{ecl-i-eta} together imply
  \eqref{ecl-Q-epz}. To do so, we need only show that $Q$ creates weak
  equivalences. Let $f\cn A \to B$ in $\cC_\la$.  The naturality
  square for $\eta$ at $f$ together with the 2-out-of-3 property imply
  that $f \in \cW_\la$ if and only if $iQf \in \cW_\la$.  Therefore, since
  $i$ creates weak equivalences, so does $Q$.

  A similar argument using naturality of $\epz$ shows
  \eqref{ecl-Q-epz} and \eqref{ecl-Q-eta} together imply
  \eqref{ecl-i-epz}.
\end{proof}

\begin{note}
  We emphasize that condition \eqref{ecl-Q-eta} does not generally
  imply the others.
\end{note}

\begin{rmk}
  In practice, we have a notion of weak equivalences in $\cC_\tau$ and
  want to extend this notion to the more general maps in $\cC_\la$ in
  a conservative way: we do not want a tight map to become a weak
  equivalence when considered as a loose map. The fact that conditions
  \eqref{ecl-i-epz} and \eqref{ecl-Q-epz} in
  \cref{thm:equiv-conditions-left-merged} are equivalent means that whenever
  $\epz$ is a weak equivalence we can achieve this by creating
  $\cW_\la$ via $Q$.  
 \end{rmk}

The same reasoning above, applied to different triangle identities,
yields the following version of \cref{thm:equiv-conditions-left-merged} when
$i$ has a right adjoint.  We will not use this version, but include it
for completeness.

\begin{thm}\label{thm:equiv-conditions-right}
  Let $(\cC_\tau, \cW_\tau)$, $(\cC_\la, \cW_\la)$ be categories with weak
  equivalences and let $\cC_\tau \fto{i} \cC_\la$ be a map extension.  Assume
  there is a right adjoint $i \dashv Q$ with counit and unit $\epz$ and
  $\eta$, respectively.  The following are equivalent:
  \begin{samepage}
  \begin{enumerate}
  \item\label{ecr-i-eta} $i$ creates weak equivalences and $\eta \in \cW_\tau$,
  \item\label{ecr-i-epz} $i$ creates weak equivalences and $\epz \in \cW_\la$,
  \item\label{ecr-Q-eta} $Q$ creates weak equivalences and $\eta \in \cW_\tau$.
  \end{enumerate}  
  Moreover, these conditions imply the following:
  \begin{enumerate}
    \setcounter{enumi}{3}
  \item\label{ecr-Q-epz} $Q$ creates weak equivalences and $\epz \in \cW_\la$.
  \end{enumerate}
    Consequently, 
  \[
  \begin{xy}
    0;<15mm,0mm>:<0mm,15mm>:: 
    (-1,0)*+{(\cC_\tau, \cW_\tau)}="ct";
    (1,0)*+{(\cC_\la, \cW_\la)}="cl";
    {\ar@/_3mm/_-{i} "ct"; "cl"};
    {\ar@/_3mm/_-{Q} "cl"; "ct"};
    (0,0)*+{\top}="v";
  \end{xy}
  \]
  is an adjoint equivalence of homotopy
  theories.
  \end{samepage}
\end{thm}

The next result shows that if weak equivalences in the category of
tight maps are detected via some underlying data, then the same is
true for the loose maps.  This is the most common situation in
examples of interest.

\begin{thm}\label{thm:u-creates}
  Assume the hypotheses and any of the equivalent statements of
  \cref{thm:equiv-conditions-left-merged}.  Furthermore, let $(\cK,\cV)$ be a
  category with weak equivalences with a commutative triangle of
  underlying categories as below.
  \[\begin{xy}
    0;<15mm,0mm>:<0mm,10mm>:: 
    (-1,0)*+{\cC_\tau}="c";
    (1,0)*+{\cC_\la}="cl";
    (0,-1)*+{\cK}="k";
    {\ar^-{i} "c"; "cl"};
    {\ar_-{U_\tau} "c"; "k"};
    {\ar^-{U_\la} "cl"; "k"};
  \end{xy}\]
  Then $U_\tau$ creates $\cW_\tau$ if and only if $U_\la$ creates $\cW_\la$.
\end{thm}
\begin{proof}
  One implication is obvious: if $U_\la$ creates weak equivalences
  then so does $U_\tau$.  Now for the converse assume that $U_\tau$
  creates weak equivalences.  We first show that $U_\la \eta \in \cV$.  
  Applying $U_\la$ to one of the triangle identities shows that
  \[
  \id_{\,U_\tau A} = U_\la i  \epz_{A} \circ U_\la \eta_{iA} = U_\tau \epz_{A} \circ U_\la \eta_{iA},
  \]
  so 2-out-of-3, the fact that $i$ is the identity on objects, and the
  assumption that $U_\tau$ creates weak equivalences shows $U_\la \eta
  \in \cV$.

  Now let $f$ be a morphism of $\cC_\la$.  Naturality of $\eta$ shows
  that $\eta \circ f = iQf \circ \eta$.  Applying $U_\la$ to this
  equation gives
  \[
  U_\la \eta \circ U_\la f = U_\la iQf \circ U_\la \eta = U_\tau Qf \circ
  U_\la \eta.
  \]
  By \cref{thm:equiv-conditions-left-merged}, $Q$ creates weak
  equivalences.  So $f \in \cW_\la$ if and only if $U_\tau Qf \in \cV$
  and, so the result follows by 2-out-of-3 and $U_\la \eta \in \cV$.
\end{proof}

\begin{defn}\label{defn:adj-equiv-hty-thy-over-K-V}
  We say that $Q \dashv i$ is an \emph{adjoint equivalence of homotopy
    theories over $(\cK, \cV)$} and write
  \[
  \begin{xy}
    0;<25mm,0mm>:<0mm,15mm>:: 
    (-1,0)*+{(\cC_\tau, \cW_\tau)}="ts";
    (1,0)*+{(\cC_\la, \cW_\la)}="tl";
    (0,-1)*+{(\cK, \cV)}="k";
    {\ar@/_3mm/_-{i} "ts"; "tl"};
    {\ar@/_3mm/_-{Q} "tl"; "ts"};
    {\ar_-{U_\tau} "ts"; "k"};
    {\ar^-{U_\la} "tl"; "k"};
    (0,0)*+{\bot}="v";
  \end{xy}
  \]
  to mean:
  \begin{enumerate}
  \item $Q \dashv i$ is an adjoint equivalence of homotopy theories,
  \item $U_\la \circ i = U_\tau$, and
  \item both $U_\la$ and $U_\tau$ create weak equivalences.  
  \end{enumerate}
  Note in particular that the triangle involving $Q$ does not
  generally commute.
\end{defn}

\section{Applications to algebras over 2-monads}
\label{sec:applications}

We will apply the results on homotopy theories in the previous section
to various categories of algebras over 2-monads.  We assume the reader
is familiar with basic 2-monad theory as developed in, e.g.,
\cite{KS74Review,BKP1989Two}.

Throughout this section we let $\cK$ be a complete and cocomplete
2-category (in the $\Cat$-enriched sense), and let $T\cn\cK \to \cK$
be a 2-monad.  Let $T\mh\Alg_s$ denote the 2-category whose 0-cells
are $T$-algebras, 1-cells are strict algebra maps, and 2-cells are
$T$-algebra transformations.  There are also notions of lax, oplax and
pseudo algebra maps, which are, respectively, the 1-cells in the
2-categories $T\mh\Alg_l$, $T\mh\Alg_{op}$, and $T\mh\Alg_{ps}$.  

In examples, $T$ might describe (symmetric)
monoidal structures, $n$-fold monoidal structures, 
diagrams in a 2-category, or $G$-equivariant structures for a group
$G$.  In the monoidal case, the four kinds of maps are:
\begin{itemize}
\item strict monoidal, with axioms like $F(x) \otimes F(y) = F(x \otimes y)$;
\item lax monoidal, with additional data like $F(x) \otimes F(y) \to
  F(x \otimes y)$, subject to new coherence axioms; 
\item oplax monoidal, with additional data like $F(x \otimes y) \to
  F(x) \otimes F(y)$, subject to the ``backwards'' version of the lax
  axioms; and
  \item strong monoidal (pseudo algebra maps), with additional data like $F(x) \otimes F(y)
  \cong F(x \otimes y)$, once again subject to new coherence axioms.
\end{itemize}

\begin{defn}[2-monadic]
  A 2-functor is called \emph{2-monadic} if it is monadic in the
  $\Cat$-enriched sense.
\end{defn}

Let $U_\om \cn T\mh\Alg_\om \to \cK$ denote any of the functors which give the
underlying objects and morphisms, where $\om$ denotes any of $s$, $l$, $op$, or $ps$.  The functor $U_s$ is then
2-monadic, and any 2-monadic functor is of this form (up to
2-equivalence of 2-categories); in particular, one should note that
2-monadicity does not capture the structure of any of the non-strict variants.

The $\Cat$-enriched monadicity theorem \cite{Dub1970Kan}
gives three essential conditions which
imply that a 2-functor $U \cn X \to Y$ is 2-monadic.
First, it must have a
left 2-adjoint.  Second, it must be conservative (see below).  Third,
$X$ must have, and $U$ must preserve, certain coequalizers. 

\begin{defn}[Conservative]
  A functor is called \emph{conservative} if it reflects isomorphisms.
\end{defn}

\begin{defn}[Accessible]

  A functor is \emph{accessible} if it preserves $\ka$-filtered
  colimits for some regular cardinal $\ka$.  A monad is called
  accessible if its underlying functor is accessible.
\end{defn}

\begin{thm}[\cite{BKP1989Two}]\label{prop:bkp-access}
  If $T$ is an accessible 2-monad on a complete and cocomplete
  2-category $\cK$, then the inclusion
  \[
  i\cn T\mh\Alg_s \hookrightarrow T\mh\Alg_l
  \]
  has a left 2-adjoint $Q$.
\end{thm}

\begin{rmk}
 The above result holds when lax is replaced with oplax or pseudo.
\end{rmk}

Bourke and Garner show in \cite{BG2014AWFSII} that $Q$ arises from an
algebraic weak factorization system using the class of lalis.  Here we
are required to use the additional power of an \emph{algebraic} weak
factorization system over the more traditional weak factorization
systems.  Algebraic refers to additional structure we require our
factorization system to possess.  Instead of having left and right
classes of maps satisfying factorization and lifting axioms, we have a
functorial factorization $f \mapsto Rf \circ Lf$ equipped with the
structure of a monad on the functor $R$ and a comonad on the functor
$L$.  The coalgebras for $L$ play the role of left maps, and the
algebras for $R$ play the role of right maps.  A (co)algebra structure
is just that: additional structure.  Thus we talk about right map
structures on a given morphism, meaning a choice of algebra structure
for the monad $R$.  While the proofs of the results quoted here depend
heavily on this extra algebraic structure, the theory of algebraic
weak factorization systems can be taken as a black box for our
purposes.  For further reading, see
\cite{GT2006Natural,BG2014AWFSI,BG2014AWFSII}.

\begin{defn}\label{defn:lali}
  A \emph{left-adjoint left-inverse}, or \emph{lali}, in $\cK$ is an
  adjunction $(f \dashv g, \epz\cn fg \Rightarrow \id, \eta\cn \id
  \Rightarrow gf)$ such that $\epz$ is the identity.
\end{defn}

\begin{prop}[\cite{BG2014AWFSII}]\label{prop:bourke-garner}
  Let $T$ be an accessible 2-monad on a complete and cocomplete 2-category $\cK$. 
  \begin{enumerate}  \item\label{it:bg-1} There is an algebraic weak factorization system
    on the underlying category of $\cK$ such that a right map
    structure on a map $f$ is a lali structure $(f \dashv g, \epz = \id, \eta)$.
  \item\label{it:bg-2} There is an algebraic weak factorization system
    on the underlying category of $T\mh\Alg_s$ such that a right map
    structure on a strict algebra map $f\cn A \to B$ in $T\mh\Alg_s$
    is a lali structure on the underlying 1-cell in $\cK$.
  \item\label{it:bg-3} The inclusion $i$ has a left adjoint $Q$
    \[\begin{xy}
      0;<25mm,0mm>:<0mm,15mm>:: 
      (0,0)*+{T\mh\Alg_s}="ts";
      (1,0)*+{T\mh\Alg_l}="tl";
      {\ar@/_3mm/_-{i} "ts"; "tl"};
      {\ar@/_3mm/_-{Q} "tl"; "ts"};
      (.5,0)*+{\bot}="v";
    \end{xy}\]
    and the counit $\epz$ of this adjunction has a right map structure as in \eqref{it:bg-2}.
  \end{enumerate}
\end{prop}

We combine the previous result with the theory of
\cref{sec:adjunctions-creating-weak-equivalences} to prove the
following.  This is the theorem we use most frequently in examples.

\begin{thm}\label{thm:algebra-rel-var-adj-equiv}
  Let $T$ be an accessible 2-monad on a complete and cocomplete
  2-category $\cK$.  Let $\cW_s$ be a collection of 1-cells which make
  the underlying 1-category of $T\mh\Alg_s$ a category with weak
  equivalences and assume $\cW_s$ contains all 1-cells $f$ such that
  $U_s f$ admits a lali structure.  Then there exists a left adjoint
  $Q$ and a unique collection of 1-cells $\cW_l$ created by $Q \dashv
  i$. Consequently,
  \[
  \begin{xy}
    0;<20mm,0mm>:<0mm,15mm>:: 
    (-1,0)*+{(T\mh\Alg_s, \cW_s)}="ts";
    (1,0)*+{(T\mh\Alg_l, \cW_l)}="tl";
    {\ar@/_3mm/_-{i} "ts"; "tl"};
    {\ar@/_3mm/_-{Q} "tl"; "ts"};
    (0,0)*+{\bot}="v";
  \end{xy}
  \]
  establishes an adjoint equivalence of homotopy theories.
\end{thm}
\begin{proof}
  This follows by combining
\cref{thm:equiv-conditions-left-merged,prop:bourke-garner}.
\end{proof}

\begin{rmk}
  There is a version of this theory that works with oplax morphisms
  instead of lax ones, and the algebraic weak factorization system
  involved uses ralis (right adjoint, left inverse) for its right maps
  instead of lalis.  Alternatively, there is a pseudo-strength
  version, using pseudomorphisms, and the corresponding algebraic weak
  factorization system is that for retract equivalences.  See
  \cite{BG2014AWFSII} for more details.  In each case we have a
  corresponding version of \cref{prop:bourke-garner} and
  \cref{thm:algebra-rel-var-adj-equiv}.
\end{rmk}

\subsection{Monads on \texorpdfstring{$\cK = \Cat$}{K = Cat}}
\label{sec:K-is-Cat}

Let $\cK = \Cat$ and let $T$ be any accessible 2-monad, for example, the
2-monad arising from an operad.  Let $(\Cat,\cV)$ be any weak
equivalence structure for which $\cV$ contains all adjunctions.
Such classes of weak equivalences arise naturally in homotopy theory.
Examples include the class of functors for which the induced map on nerves
is a weak homotopy equivalence and the class of functors for which the
induced map on nerves is an $E$-(co)homology isomorphism 
for some spectrum $E$.  

Let $\cW_s$ be the weak equivalence structure on $T\mh\Alg_s$ created
by
\[
U_s\cn T\mh\Alg_s \to \Cat
\]
and let $\om$ be any of $l$, $op$, or $ps$.  Then, by the appropriate
variant of \cref{thm:algebra-rel-var-adj-equiv}, the category
$T\mh\Alg_\om$ has the weak equivalence structure created by $Q \dashv
i$ and we have an adjoint equivalence of homotopy theories.  The
hypothesis that $\cV$ contains all adjunctions ensures that the counit
$\epz$ is a weak equivalence.  By \cref{thm:u-creates} this is also
the weak equivalence structure created by the forgetful functor
$U_\om\cn T\mh\Alg_\om \to \Cat$.

\begin{notn}
  Let $\mathit{we}$ denote the class of weak homotopy equivalences in
  $\Cat$, i.e., those functors which induce a weak homotopy
  equivalence on nerves.  We abusively use this notation for any class
  of weak equivalences created by a functor to $(\Cat,
  \,\mathit{we})$.
\end{notn}

\begin{example}[Symmetric monoidal categories]\label{eg:sym-mon-cat}
  The prototypical example of this kind is when $T$ is the 2-monad for
  symmetric monoidal categories.  Then $T\mh\Alg_s$ is the 2-category
  of symmetric monoidal categories, symmetric strict monoidal
  functors, and monoidal transformations, while $T\mh\Alg_l$ has the
  same objects but symmetric lax monoidal functors.  Let $\cV =
  \mathit{we}$ and let the underlying category functor $\SMC_s \to
  \Cat$ create weak equivalences.  By
  \cref{thm:algebra-rel-var-adj-equiv} we have the following adjoint
  equivalence of homotopy theories over $(\Cat, \mathit{we})$.
  \[
  \begin{xy}
    0;<25mm,0mm>:<0mm,15mm>:: 
    (-1,0)*+{(\SMC_s,\, \mathit{we})}="ts";
    (1,0)*+{(\SMC_l,\, \mathit{we})}="tl";
    (0,-1)*+{(\Cat, \mathit{we})}="k";
    {\ar@/_3mm/_-{i} "ts"; "tl"};
    {\ar@/_3mm/_-{Q} "tl"; "ts"};
    {\ar_-{U_s} "ts"; "k"};
    {\ar^-{U_l} "tl"; "k"};
    (0,0)*+{\bot}="v";
  \end{xy}
  \]
  As noted above, one also has pseudo and oplax variants of this
  example which likewise give adjoint equivalences of homotopy
  theories.  The pseudo algebra maps in this case are the strong
  symmetric monoidal maps.
\end{example}

\begin{example}[Symmetric monoidal categories and normal functors]\label{eg:sym-mon-cat-nor}
  A slight variant of our first example uses a different 2-monad $T$
  on $\Cat_*$ whose algebras are still symmetric monoidal categories.
  In this case, the specified base point becomes the unit object of the
  symmetric monoidal structure.  The category $T\mh\Alg_s$ consists of
  symmetric monoidal categories and symmetric strict monoidal
  functors, while $T\mh\Alg_l$ is now the category of symmetric
  monoidal categories and \emph{normal} (i.e., strictly unit
  preserving) symmetric lax monoidal functors.  We take $\cV =
  \mathit{we}$ in $\Cat_{*}$ to be the class of unbased weak homotopy
  equivalences (created by the forgetful functor to $\Cat$).  By
  \cref{thm:algebra-rel-var-adj-equiv} we have the following adjoint
  equivalence of homotopy theories over $(\Cat_*,\, \mathit{we})$.
  \[
  \begin{xy}
    0;<25mm,0mm>:<0mm,15mm>:: 
    (-1,0)*+{(\SMC_s,\, \mathit{we})}="ts";
    (1,0)*+{(\SMC_{nl},\, \mathit{we})}="tl";
    (0,-1)*+{(\Cat_*, \mathit{we})}="k";
    {\ar@/_3mm/_-{i} "ts"; "tl"};
    {\ar@/_3mm/_-{Q} "tl"; "ts"};
    {\ar_-{U_s} "ts"; "k"};
    {\ar^-{U_l} "tl"; "k"};
    (0,0)*+{\bot}="v";
  \end{xy}
  \]
\end{example}

We also have oplax and pseudo variants of the previous examples.
Combining these yields the following strengthening of
\cite[3.9]{Man10Inverse}.
\begin{thm}\label{thm:sym-mon-cat-we}
  The homotopy theory of $(\SMC_{s},\, \mathit{we})$ is equivalent to
  the homotopy theory of each of the following.
  \vspace{-1pc} 
  \setlength{\linewidth}{.7\linewidth}
  \begin{multicols}{2}
  \begin{itemize}
    \item  $(\SMC_{ps},\, \mathit{we})$
    \item  $(\SMC_{l},\, \mathit{we})$
    \item  $(\SMC_{op},\, \mathit{we})$
    \item  $(\SMC_{nps},\, \mathit{we})$
    \item  $(\SMC_{nl},\, \mathit{we})$
    \item  $(\SMC_{nop},\, \mathit{we})$
  \end{itemize}
  \end{multicols}
\end{thm}

\begin{example}[Stable equivalences of symmetric monoidal categories]\label{eg:sym-mon-cat-nop}
  For a final variant concerning symmetric monoidal categories, we
  take the normal, oplax version of the above example.  For the
  ``underlying'' category we now take the category of
  $\Ga$-categories, with $\cV = \mathit{st~eq}$ being the class of
  stable equivalences \cite{BF1978}.  This example differs from the
  previous ones in that we do not know whether $K$ satisfies
  monadicity and therefore cannot apply
  \cref{thm:algebra-rel-var-adj-equiv}.  However we can apply
  \cref{thm:equiv-conditions-left-merged,thm:u-creates} directly.  Let
  $U_\la = K$ be the $K$-theory functor for normal, oplax symmetric
  monoidal functors from \cite{Man10Inverse}, and let $U_\tau = K$ be
  the restriction to strict functors.
  
  Let $\cW_s = \mathit{st~eq}$ be the weak equivalences created by
  $K$.  Then the left adjoint $Q$ arises as in the previous examples
  but via the oplax variant of \cref{prop:bkp-access}.  We therefore
  have the following adjoint equivalence of homotopy theories over
  $(\Ga\mh\Cat, \, \mathit{st~eq})$.
  \[
  \begin{xy}
    0;<25mm,0mm>:<0mm,15mm>:: 
    (-1,0)*+{(\SMC_s, \, \mathit{st~eq})}="ts";
    (1,0)*+{(\SMC_{nop}, \, \mathit{st~eq})}="tl";
    (0,-1)*+{(\Ga\mh\Cat, \, \mathit{st~eq})}="k";
    {\ar@/_3mm/_-{i} "ts"; "tl"};
    {\ar@/_3mm/_-{Q} "tl"; "ts"};
    {\ar_-{K} "ts"; "k"};
    {\ar^-{K} "tl"; "k"};
    (0,0)*+{\bot}="v";
  \end{xy}
  \]
\end{example}

Once again we can consider other variants, making use of the
alternative definitions of $K$ given in \cite{Man10Inverse}.  Together
these give the following generalization of
\cite{Tho95Symmetric,Man10Inverse}.
\begin{thm}\label{thm:sym-mon-cat-st-eq}
  The homotopy theory of $(\SMC_{s},\, \mathit{st~eq})$ is equivalent
  to the homotopy theory of each of the following.
  \vspace{-1pc} 
  \setlength{\linewidth}{.7\linewidth}
  \begin{multicols}{2}
  \begin{itemize}
    \item  $(\SMC_{ps},\, \mathit{st~eq})$
    \item  $(\SMC_{l},\, \mathit{st~eq})$
    \item  $(\SMC_{op},\, \mathit{st~eq})$
    \item  $(\SMC_{nps},\, \mathit{st~eq})$
    \item  $(\SMC_{nl},\, \mathit{st~eq})$
    \item  $(\SMC_{nop},\, \mathit{st~eq})$
  \end{itemize}
  \end{multicols}
\end{thm}

Our next examples concern $n$-fold monoidal categories for $n \ge 1$.
These were introduced by Balteanu-Fiedorowicz-Schw\"anzl-Vogt
\cite{BFSV2003Iterated} and are the algebras over an operad $\sM_n$
whose geometric realization is equivalent to the little $n$-cubes
operad.

Alternatively, an $n$-fold monoidal category can be defined
iteratively as a monoid in the category
$(n-1)\mh\mathpzc{Mon}\Cat_{nl}$ of $(n-1)$-fold monoidal categories
and normal lax monoidal maps.  Laxity of the monoid structure map
gives rise to interchange maps between the $n$ different monoidal
products and also to a hexagonal interchange axiom.  A 1-fold monoidal
category is simply a monoidal category.  The notion of braided
monoidal category is equivalent to that of a 2-fold monoidal category
where both products are the same and their interchange transformation
is invertible.

To apply our general theory we must identify the lax maps of $n$-fold
monoidal categories as the lax algebra maps for the 2-monad associated
to $\sM_n$.  This does not appear in the literature, but follows from
Bourke's 2-dimensional monadicity (see \cref{sec:2D-monadicity}).
  
\begin{prop}\label{prop:n-MonCat-lax-is-algebras}
  Let $n \ge 1$.  The 2-category $n\mh\MonCat_l$ of $n$-fold monoidal
  categories and lax maps is 2-equivalent to the 2-category of
  algebras and lax algebra maps associated to the 2-monad $\sM_n$.
\end{prop}
We prove \cref{prop:n-MonCat-lax-is-algebras} in \cref{sec:applications-of-Bourke-monadicity}.

\begin{example}[Iterated monoidal categories]\label{eg:n-mon}
  Let $T$ be the 2-monad on $\Cat$ associated to the operad $\sM_n$.
  Let $\cV = \mathit{we}$ and let $\cW_s$ be created by the forgetful
  functor.  By
  \cref{prop:n-MonCat-lax-is-algebras,thm:algebra-rel-var-adj-equiv}
  we have the following adjoint equivalence of homotopy theories over
  $(\Cat, \,\mathit{we})$.
  \[
  \begin{xy}
    0;<23mm,0mm>:<0mm,15mm>:: 
    (-1,0)*+{(n\mh\MonCat_s, \, \mathit{we})}="ts";
    (1,0)*+{(n\mh\MonCat_l,\, \mathit{we})}="tl";
    (0,-1)*+{(\Cat, \, \mathit{we})}="k";
    {\ar@/_3mm/_-{i} "ts"; "tl"};
    {\ar@/_3mm/_-{Q} "tl"; "ts"};
    {\ar_-{U_s} "ts"; "k"};
    {\ar^-{U_l} "tl"; "k"};
    (0,0)*+{\bot}="v";
  \end{xy}
  \]
  One also has a normal version of the previous example, similar to 
  \cref{eg:sym-mon-cat-nor}.  The lax algebra maps in this case
  correspond to the $n$-fold monoidal category maps considered in
  \cite{BFSV2003Iterated}.
\end{example}

\begin{example}[Group completion equivalences of iterated monoidal categories]\label{eg:n-mon-gc}
  Let $\cV = \mathit{gc\;eq}$ be the class of weak equivalences upon
  group completion of nerves.  For $n = 1$ this means equivalence
  after applying $\Om B$ as discussed in
  \cite{Seg74Categories,BFSV2003Iterated}.  Applying
  \cref{thm:equiv-conditions-left-merged} we have the following
  adjoint equivalence of homotopy theories.
    \[
  \begin{xy}
    0;<23mm,0mm>:<0mm,15mm>:: 
    (-1,0)*+{(n\mh\MonCat_s, \, \mathit{gc\;eq})}="ts";
    (1,0)*+{(n\mh\MonCat_{l},\, \mathit{gc\;eq})}="tl";
    {\ar@/_3mm/_-{i} "ts"; "tl"};
    {\ar@/_3mm/_-{Q} "tl"; "ts"};
    (0,0)*+{\bot}="v";
  \end{xy}
  \]
\end{example}

\subsection{Diagrams in a 2-category}
\label{sec:diagrams}

Let $I$ be a small 2-category and let $\cK$ be a complete and
cocomplete 2-category.  Let $\ob I$ denote the discrete 2-category
with the same objects as $I$.  The inclusion $\ob I \to I$ induces a
2-functor
\[
U \cn [I,\cK] \to [\ob I, \cK],
\]
where $[-,-]$ denotes the 2-category of 2-functors, 2-natural
transformations, and modifications.

This $U$ has left and right adjoints given by Kan extensions and is
conservative so is 2-monadic.  The associated 2-monad $T = U \circ
\Lan$ can be computed using a coend formula.  Now since $U$ has both
adjoints it preserves all limits and colimits.  The left Kan extension
is a left adjoint so preserves colimits.  Therefore $T$ preserves all
colimits and hence is accessible.  We summarize this discussion in
\cref{prop:i-monadic}.

\begin{prop}\label{prop:i-monadic}
  The 2-functor
  \[
  U\cn [I,\cK] \monoto [\ob I, \cK]
  \]
  is 2-monadic and the associated 2-monad is accessible.
\end{prop}

The next result recognizes $[I,\cK]_l$ as the 2-category of algebras and lax maps over $T$.

\begin{prop}\label{prop:diagrams}
  Let $I$, $\cK$, and $T$ be as above.  Then
  $T\mh\Alg_l \iso [I,\cK]_l$,
  the 2-category of diagrams, lax transformations, and
  modifications.
\end{prop}
We prove \cref{prop:diagrams} simultaneously with a reduced version,
\cref{prop:red-diagrams}, using Bourke's 2-dimensional monadicity in
\cref{sec:applications-of-Bourke-monadicity}.  One also has pseudo and
oplax versions of \cref{prop:diagrams} giving $T\mh\Alg_{ps} \iso
[I,\cK]_{ps}$ and $T\mh\Alg_{op} \iso [I,\cK]_{op}$.  These facts are
well-known in the 2-categorical literature \cite[Example
6.6]{BKP1989Two} and follow from a straightforward, if uninteresting,
calculation.  The reduced version does not appear in the literature to
our knowledge, but is the case of interest for topological
applications.

Our next examples concern categories with group actions.  Let $G$ be a
discrete group, and let $I = BG$ be the small category with one object
having automorphisms given by $G$ and let $\cK = \Cat$.  A diagram $BG
\to \Cat$ is precisely a category with a $G$-action, and strict
diagram maps are $G$-equivariant functors.  For any morphism variant
$\om$ we denote $G\mh\Cat_\om = [BG, \Cat]_\om$.  Note that $[\ob BG,
  \Cat]$ is $\Cat$.

  The pseudonatural maps are functors that preserve equivariance only
  up to coherent natural isomorphisms. These appear naturally in
  examples, for instance, in Merling's study of equivariant algebraic
  $K$-theory \cite{Mer2017Equivariant}. There are several notions of
  weak equivalence in the category of small $G$-categories and strict
  equivariant maps that are of interest to algebraic topologists and
  we discuss these below.
  
\begin{example}[$G$-categories with underlying weak equivalences]\label{eg:g-cat}
  In this example we consider $G\mh\Cat$ with weak equivalences being
  the equivariant functors that are weak homotopy equivalences on
  underlying categories.  Let $\mathit{we}$ denote this class.
  Combining \cref{prop:i-monadic,prop:diagrams} with
  \cref{thm:algebra-rel-var-adj-equiv}, we have the following adjoint
  equivalence of homotopy theories over $(\Cat,\,\mathit{we})$.
  \[
  \begin{xy}
    0;<23mm,0mm>:<0mm,15mm>:: 
    (-1,0)*+{(G\mh\Cat_s, \, \mathit{we})}="ts";
    (1,0)*+{(G\mh\Cat_{ps},\, \mathit{we})}="tl";
    (0,-1)*+{(\Cat, \, \mathit{we})}="k";
    {\ar@/_3mm/_-{i} "ts"; "tl"};
    {\ar@/_3mm/_-{Q} "tl"; "ts"};
    {\ar_-{U_s} "ts"; "k"};
    {\ar^-{U_l} "tl"; "k"};
    (0,0)*+{\bot}="v";
  \end{xy}
  \]
\end{example}
  
\begin{nonexample}[$G$-categories with $G$-weak equivalences]
  Let $G\mh\mathit{we}$ denote the class of 
  $G$-weak equivalences, i.e., the equivariant functors $F\cn \cC\to
  \cD$ that induce weak homotopy equivalences on fixed points $\cC^H \to \cD^H$
  for all subgroups $H$ of $G$ \cite{may1996alaska,BMOOPY}. These are
  the weak equivalences of primary homotopical interest.

  The counit of the adjunction
  \[
  \begin{xy}
    0;<15mm,0mm>:<0mm,15mm>:: 
    (-1,0)*+{G\mh\Cat_s}="ts";
    (1,0)*+{G\mh\Cat_{ps}}="tl";
    {\ar@/_3mm/_-{i} "ts"; "tl"};
    {\ar@/_3mm/_-{Q} "tl"; "ts"};
    (0,0)*+{\bot}="v";
  \end{xy}
  \]
  is not a $G$-weak equivalence unless $G$ is trivial. Indeed, for the
  terminal category $*$ with trivial $G$-action, $Q(*)=EG$, the
  category with set of objects equal to $G$ and a unique morphism
  between any two objects.  The action of $G$ on $EG$ is given by left
  multiplication. The counit $EG \to *$ is a non-equivariant weak homotopy
  equivalence but not a $G$-weak equivalence since $EG^H$ is the empty
  category for all nontrivial subgroups.
  This means that $Q \dashv i$ is not an adjoint equivalence of
  homotopy theories between $(G\mh\Cat_s, \, G\mh\mathit{we})$ and
  $(G\mh\Cat_{ps},\cW_{ps})$ for \emph{any} choice of $\cW_{ps}$.
\end{nonexample}

\subsection{Reduced diagrams in a 2-category}
\label{sec:reduced-diagrams}

\begin{note}
  We remind the reader that all limits and colimits are to be
  interpreted in the $\Cat$-enriched sense.  Thus a terminal object
  $*$ in $\cK$ is one such that $\cK(x,*)$ is the terminal category
  for all objects $x$.
\end{note}

In this section, let $I$ be a small 2-category with a zero object 0
and let $\cK$ be a complete and cocomplete 2-category with terminal
object $*$.
\begin{defn}
  A \emph{reduced diagram} is a 2-functor $X\cn I \to \cK$ such that
  $X(0) \cong *$.
\end{defn}

Let $[I,\cK]_{\mathrm{red}}$ denote the 2-category of reduced
2-functors, all 2-natural transformations, and modifications.  Let
\[
j\cn [I,\cK]_{\mathrm{red}} \monoto [I,\cK]
\]
denote the inclusion of reduced diagrams into all diagrams.
We define 
\[
R\cn [I,\cK] \to [I,\cK]_{\mathrm{red}} 
\]
using a quotient construction as follows.  If $X$ is any diagram and
$a \in I$, define $RX(a)$ by the pushout below. 
\[
\begin{xy}
  0;<20mm,0mm>:<0mm,15mm>:: 
  (0,0)*+{X(0)}="x0";
  (1,0)*+{X(a)}="xa";
  (0,-1)*+{*}="pt";
  (1,-1)*+{RX(a)}="rxa";
  {\ar^-{X(!)} "x0"; "xa"};
  {\ar "pt"; "rxa"};
  {\ar "x0"; "pt"};
  {\ar "xa"; "rxa"};
\end{xy}
\]
This levelwise pushout is a pushout in the 2-category $[I,\cK]$ as we
now explain.  Let $c_{0}X$ be the constant diagram on $X(0)$.  The
maps $X(0 \to a)$ for $a \in I$ are the components of a map of
diagrams $c_0 X \to X$.  The diagram $RX$ is then the pushout below.
\[
\begin{xy}
  0;<20mm,0mm>:<0mm,15mm>:: 
  (0,0)*+{c_{0}X}="x0";
  (1,0)*+{X}="xa";
  (0,-1)*+{*}="pt";
  (1,-1)*+{RX}="rxa";
  {\ar "x0"; "xa"};
  {\ar "pt"; "rxa"};
  {\ar "x0"; "pt"};
  {\ar "xa"; "rxa"};
\end{xy}
\] 
This construction is reduced because a pushout along an isomorphism is
an isomorphism.  The universal property of the 2-categorical pushout
shows that $R$ is a 2-functor $[I,\cK] \to [I,\cK]_{\mathrm{red}}$.

\begin{prop}
  The construction $R$ above is left 2-adjoint to the inclusion 
  \[
  j\cn [I,\cK]_{\mathrm{red}} \monoto [I,\cK].
  \]
\end{prop}
\begin{proof}

  For any diagram $X$ we have
  1-cells
  \[
  X(a) \to RX(a),
  \]
  which form a 2-natural transformation
  \[
  \eta_X \cn X \to jRX
  \]
  by the 2-dimensional nature of the universal property of the
  pushout. As $X$ varies in $[I,\cK]$, these assemble into a 2-natural
  transformation $\eta$ from the identity to $jR$.
  
  If $Y$ is reduced, these 1-cells are isomorphisms because each
  pushout along the isomorphism $Y(0) \iso *$ is an isomorphism.
  Their inverses give a 2-natural transformation
  \[
  \epz_Y \cn RjY \to Y,
  \]
  which will assemble into a 2-natural transformation from $Rj$ to the
  identity 2-functor on $[I,\cK]_{\mathrm{red}}$.  One of the triangle
  identities is immediate from the definition of $\epz$. To prove the other, that
  \[
  \begin{xy}
    0;<20mm,0mm>:<0mm,15mm>:: 
    (0,0)*+{RX}="rx";
    (1,0)*+{RjRX}="rjrx";
    (1,-1)*+{RX}="rx2";
    {\ar^{R\eta_X} "rx"; "rjrx"};
    {\ar@{=} "rx"; "rx2"};
    {\ar^-{\epz_{RX}} "rjrx"; "rx2"};
  \end{xy}
  \] 
  commutes, we prove that for an object $a$ of $I$, the 1-cells
  $(R\eta_X)_a$ and $(\eta_{RX})_a$ are equal. Now $(R\eta_X)_a$ is
  defined as the unique dotted 1-cell that makes the diagram
  \[
  \begin{xy}
    0;<20mm,0mm>:<0mm,20mm>:: 
    (0,0)*+{X(0)}="x0";
    (1,0)*+{X(a)}="xa";
    (0,-1)*+{*}="pt";
    (1,-1)*+{RX(a)}="rxa";
    (1.7,-.7)*+{RX(a)}="xb";
    (0.7,-1.7)*+{*}="pt2";
    (1.7,-1.7)*+{RRX(a)}="rxb";
    {\ar^-{X(!)} "x0"; "xa"};
    {\ar "pt"; "rxa"};
    {\ar "x0"; "pt"};
    {\ar "pt2"; "rxb"};
    {\ar@{=} "pt"; "pt2"};
    {\ar^-{(\eta_{RX})_a} "xb"; "rxb"};
    {\ar^-{(\eta_X)_a} "xa"; "xb"};
    {\ar_-{(\eta_X)_a} "xa"; "rxa"};
    {\ar@{-->}_-{(R\eta_X)_a} "rxa"; "rxb"}; 
  \end{xy}
  \] 
  commute, but $(\eta_{RX})_a$ is also such a morphism so $(R\eta_X)_a
  = (\eta_{RX})_a$.  This verifies the second triangle identity, since
  $\epz$ was defined as the inverse of $\eta$.
\end{proof}

\begin{prop}\label{prop:j-monadic}
  Let $\cK$ be a complete and cocomplete 2-category.  
  Then the inclusion
  \[
  j\cn [I,\cK]_{\mathrm{red}} \monoto [I,\cK]
  \]
  is 2-monadic and the associated 2-monad is accessible.
\end{prop}
\begin{proof}
  We apply the $\Cat$-enriched version of Beck's
  monadicity theorem \cite{Dub1970Kan}.  To do this, we have only to check that $j$
  is conservative, has a left adjoint, and preserves certain
  coequalizers.  We have already constructed the left adjoint above.
  Isomorphisms of diagrams are levelwise isomorphisms in both
  categories, so $j$ is conservative.  One can easily verify that
  coequalizers in $[I,\cK]_{\mathrm{red}}$ are computed levelwise,
  hence they exist and $j$ preserves all of them.  The same is also
  true for filtered colimits (in fact all connected colimits), so $j$
  preserves them and therefore the associated 2-monad is accessible.
\end{proof}

The 2-category $[I,\cK]_{\mathrm{red}}$ is also complete and
cocomplete as a 2-category since it is a full reflective 2-category of
a complete and cocomplete 2-category: limits are computed levelwise,
and colimits are computed by first applying $j$, then taking the
colimit in $[I,\cK]$, and then applying $R$ to get a reduced
diagram. For the remainder of this section we let $U$ denote the
composition of $j$ with pullback along the inclusion $\ob I \monoto
I$.  The same arguments as above prove the next result.

\begin{prop}\label{prop:ij-monadic}
  The composite
  \[
  U \cn [I,\cK]_{\mathrm{red}} \monoto [I,\cK] \monoto [\ob I, \cK]
  \]
  is 2-monadic and the associated 2-monad is accessible.
\end{prop}
Let $T$ be the composite $U \circ R \circ \Lan$, the 2-monad
associated with the composite adjunction.  Then $T\mh\Alg_s$ is $[I,\cK]_{\mathrm{red}}$.
\begin{prop}\label{prop:red-diagrams}
  The 2-category $T\mh\Alg_l$ is $[I,\cK]_{\mathrm{red},l}$, the
  2-category of reduced diagrams, lax transformations,
  and modifications.
\end{prop}
We prove \cref{prop:red-diagrams} using Bourke's 2-dimensional
monadicity in \cref{sec:applications-of-Bourke-monadicity}.
 
\begin{example}[$\Ga$-objects and levelwise weak equivalences]\label{eg:gamma-objects}
  Let $I$ be a skeleton of the category of finite based sets
  considered as a discrete 2-category and let $\cK$ be either $\Cat$
  or $\IICat_2$, the 2-category of 2-categories, 2-functors, and
  2-natural transformations.  Then $[I,\Cat]_{\mathrm{red}}$ is the
  2-category of $\Ga$-categories, $\Ga$-functors, and
  $\Ga$-transformations \cite{Seg74Categories}, while
  $[I,\IICat_{2}]_{\mathrm{red}}$ is the 2-category of
  $\Ga$-2-categories, $\Ga$-2-functors, and $\Ga$-transformations
  studied in \cite{GJO2017KTheory}.

  By \cref{prop:diagrams,prop:red-diagrams}, these are 2-monadic over
  $[\ob I, \Cat] = [\mathbb{N}, \Cat]$ and, respectively,
  $[\mathbb{N}, \IICat_{2}]$.  The corresponding 2-categories with lax
  algebra maps are, respectively, $[I,\Cat]_{\mathrm{red},l}$ and
  $[I,\IICat_{2}]_{\mathrm{red},l}$.  These are the 2-categories of
  $\Ga$-(2-)categories, $\Ga$-lax (2-)functors, and
  $\Ga$-transformations.  Since the 2-monads for these are accessible,
  we get a 2-adjunction $Q \dashv i$ as in \cref{prop:bourke-garner}.
  The counit of this adjunction is a right map as in
  \cref{prop:bourke-garner}~\eqref{it:bg-2}, and in particular a
  levelwise left adjoint.

  Let $\cV = \mathit{we}$ be the class of levelwise weak homotopy
  equivalences, created by the forgetful functors to $[\mathbb{N},
  \Cat]$ and $[\mathbb{N}, \IICat_{2}]$, respectively.  Being a
  levelwise adjoint, the counit is a levelwise equivalence, so by
  \cref{thm:equiv-conditions-left-merged,thm:u-creates} we have the
  following adjoint equivalences of homotopy theories.
  \[
  \begin{xy}
    0;<72.5mm,0mm>:<0mm,15mm>:: 
  (0,0)*{
  \begin{xy}
    0;<20mm,0mm>:<0mm,15mm>:: 
    (-1,0)*+{(\Ga\mh\Cat_s,\, \mathit{we})}="ts";
    (1,0)*+{(\Ga\mh\Cat_l,\, \mathit{we})}="tl";
    (0,-1)*+{([\mathbb{N}, \Cat],\, \mathit{we})}="k";
    {\ar@/_3mm/_-{i} "ts"; "tl"};
    {\ar@/_3mm/_-{Q} "tl"; "ts"};
    {\ar_-{U} "ts"; "k"};
    {\ar^-{U} "tl"; "k"};
    (0,0)*+{\bot}="v";
  \end{xy}
  };
  (1,0)*{
  \begin{xy}
    0;<20mm,0mm>:<0mm,15mm>:: 
    (-1,0)*+{(\Ga\mh\IICat_s,\, \mathit{we})}="ts";
    (1,0)*+{(\Ga\mh\IICat_l,\, \mathit{we})}="tl";
    (0,-1)*+{([\mathbb{N}, \IICat_2],\, \mathit{we})}="k";
    {\ar@/_3mm/_-{i} "ts"; "tl"};
    {\ar@/_3mm/_-{Q} "tl"; "ts"};
    {\ar_-{U} "ts"; "k"};
    {\ar^-{U} "tl"; "k"};
    (0,0)*+{\bot}="v";
  \end{xy}
  }
  \end{xy}
  \]
  Note that this is a stronger result
  than what we were able to achieve using direct methods in
  \cite{GJO2017KTheory}, namely it strengthens Theorem 4.37 and
  Corollary 4.47 of \textit{loc. cit.} to equivalences of homotopy
  theories rather than just of homotopy categories.
\end{example}

\begin{example}[$\Ga$-objects and stable equivalences]\label{ex:gamma-st-equiv}
  We can also consider the class of stable equivalences, $\cV =
  \mathit{st\;eq}$, which are now created by the levelwise nerve
  functor to $\Ga\mh\sSet$ \cite{BF1978}.  The counit of
  the adjunction between the categories with strict and lax maps is a
  stable equivalence since, once again, it is a levelwise equivalence.  By
  \cref{thm:equiv-conditions-left-merged} we have the following
  adjoint equivalences of homotopy theories.
  \[
  \begin{xy}
    0;<72.5mm,0mm>:<0mm,15mm>:: 
  (0,0)*{
  \begin{xy}
    0;<20mm,0mm>:<0mm,15mm>:: 
    (-1,0)*+{(\Ga\mh\Cat_s,\, \mathit{st\;eq})}="ts";
    (1,0)*+{(\Ga\mh\Cat_l,\, \mathit{st\;eq})}="tl";
    {\ar@/_4.5mm/_-{i} "ts"; "tl"};
    {\ar@/_4.5mm/_-{Q} "tl"; "ts"};
    (0,0)*+{\bot}="v";
  \end{xy}
  };
  (1,0)*{
  \begin{xy}
    0;<20mm,0mm>:<0mm,15mm>:: 
    (-1,0)*+{(\Ga\mh\IICat_s,\, \mathit{st\;eq})}="ts";
    (1,0)*+{(\Ga\mh\IICat_l,\, \mathit{st\;eq})}="tl";
    {\ar@/_4.5mm/_-{i} "ts"; "tl"};
    {\ar@/_4.5mm/_-{Q} "tl"; "ts"};
    (0,0)*+{\bot}="v";
  \end{xy}
  };
  \end{xy}
  \]
  This is a strengthening of \cite[Corollary 4.49]{GJO2017KTheory}.
\end{example}

\section{2-dimensional monadicity}
\label{sec:2D-monadicity}

In this section we recall and apply the 2-dimensional monadicity of
Bourke \cite{Bou2014Two}.  This goes beyond elementary $\Cat$-enriched
monadicity as it accounts simultaneously for both strict and lax
algebra maps.  This enables us to identify the bare-handed notions of
lax morphisms as the lax algebra morphisms for iterated monoidal
categories (\cref{eg:n-mon,eg:n-mon-gc}) and for $\Ga$-(2-)categories
(\cref{eg:gamma-objects,ex:gamma-st-equiv}).

\subsection{The 2-dimensional monadicity theorem for lax maps}
\label{sec:2D-monadicity-theorem}

Throughout this section, let $\cB$ be a 2-category and let $j\cn
\cA_\tau \to \cA_\la$ be a 2-functor over $\cB$ via 2-functors
$H_\tau\cn \cA_\tau \to \cB$ and $H_\la\cn \cA_\tau \to \cB$ as below.
\begin{equation}\label{eq:2d-monadicity-setup}
  \begin{xy}
    0;<15mm,0mm>:<0mm,15mm>:: 
    (-1,0)*+{\cA_\tau}="at";
    (1,0)*+{\cA_\la}="ala";
    (0,-1)*+{\cB}="b";
    {\ar^{j} "at"; "ala"};
    {\ar_-{H_\tau} "at"; "b"};
    {\ar^-{H_\la} "ala"; "b"};
  \end{xy}
\end{equation} 
We further assume that $j$ is:
\begin{enumerate}
\item the identity on objects,
\item locally full and faithful on 2-cells, and
\item faithful on 1-cells.
\end{enumerate}
In particular $j$ induces a map extension on underlying 1-categories.
This is the notion of $\cF$-category introduced in
\cite{LS2012Enhanced}.  We will often suppress the subscripts on $H$
as they are clear from context.

\begin{rmk}
  For the remainder of this section we restrict to considering the
  strict/lax case.  Analogous versions of the theory for
  strict/pseudo and for strict/oplax can be found in \cite{Bou2014Two}.
\end{rmk}

\begin{defn}[strict/lax monadic]\label{defn:sl-monadic}
  We say that the pair $(H_\tau, H_\la)$ is \emph{strict/lax monadic}
  if there are 2-equivalences
  \[
  \cA_\tau \hty T\mh\Alg_s
  \]
  \[
  \cA_\la \hty T\mh\Alg_l
  \]
  over $\cB$ for some 2-monad $T$ on $\cB$ such that
	\[
  \begin{xy}
    0;<28mm,0mm>:<0mm,13mm>:: 
    (0,0)*+{\cA_\tau}="at";
    (0,-1)*+{T\mh\Alg_s}="ts";
    (1,0)*+{\cA_\la}="al";
    (1,-1)*+{T\mh\Alg_l}="tl";
    {\ar "at"; "ts"};
    {\ar_{j} "ts"; "tl"};
    {\ar^{j} "at"; "al"};
    {\ar "al"; "tl"};		
  \end{xy}
  \]
  commutes.
\end{defn}

Conditions for a given pair to be strict/lax monadic will be given
below, and rely on the following definitions.  Note that we have
suppressed the inclusion $j$ in what follows.

\begin{defn}[Colax limit]\label{defn:colax-lim}
  Given $f\cn A \to B$ in $\cA_\la$, the \emph{colax limit} of $f$
  consists of 1-cells $p_f$ and $q_f$ in $\cA_\tau$ and a 2-cell 
    $\si_f$ in $\cA_\la$
  \[\begin{xy}
    0;<15mm,0mm>:<0mm,15mm>:: 
    (0,1)*+{C_f}="cf";
    (-1,0)*+{A}="a";
    (1,0)*+{B}="b";
    {\ar_-{p_f} "cf"; "a"};
    {\ar^-{q_f} "cf"; "b"};
    {\ar_-{f} "a"; "b"};    
    {\ar@{=>}^-{\si_f} (.1,.5)="w"; "w"+(-.2,-.2) };
  \end{xy}\] 
  such that the following conditions hold.
  \begin{enumerate}
  \item Given 1-cells $r\cn X \to A$, $s\cn X \to B$ in $\cA_\la$ and
    a 2-cell $\al\cn s \Rightarrow fr$ in $\cA_\la$ as shown below,
    there is a unique $t \in \cA_\la$ giving the indicated equalities.
    
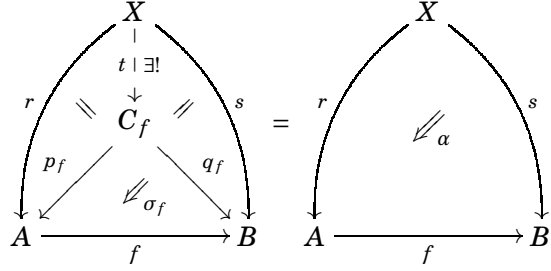
\[\begin{xy}
      0;<15mm,0mm>:<0mm,15mm>:: 
      (0,1)*+{X}="x";
      (0,0)*+{C_f}="cf";
      (-1,-1)*+{A}="a";
      (1,-1)*+{B}="b";
      {\ar@/_5mm/_-{r} "x"; "a"};
      {\ar@/^5mm/^-{s} "x"; "b"};
      {\ar@{-->}^-{\exists !}_-{t} "x"; "cf"};
      {\ar_-{p_f} "cf"; "a"};
      {\ar^-{q_f} "cf"; "b"};
      {\ar_-{f} "a"; "b"};    
      {\ar@{=} (0,0)+(.5,.2)="w"; "w"+<-2mm,-2mm>};
      {\ar@{=} (0,0)+(-.5,.2)="w"; "w"+<2mm,-2mm>};
      {\ar@{=>}^-{\si_f} (.1,-.5)="w"; "w"+(-.2,-.2) };
    \end{xy}
    =
    \begin{xy}
      0;<15mm,0mm>:<0mm,15mm>:: 
      (0,1)*+{X}="x";
      (-1,-1)*+{A}="a";
      (1,-1)*+{B}="b";
      {\ar@/_5mm/_-{r} "x"; "a"};
      {\ar@/^5mm/^-{s} "x"; "b"};
      {\ar_-{f} "a"; "b"};    
      {\ar@{=>}^-{\al} (.125,.1)="w"; "w"+(-.25,-.25)};
    \end{xy}\]
    
  \item Let $(r,s,\al)$ and $(r',s',\al')$ be as above and let $\theta_r\cn
    r \Rightarrow r'$, $\theta_s\cn s \Rightarrow s'$ be such that 
    $\al' \theta_s = (f*\theta_r) \al$.  Then there is a unique
    $\theta_t\cn t \Rightarrow t'$ such that $ p_f * \theta_t = \theta_r$
    and $q_f * \theta_t = \theta_s$.   

  \item The structure 1-cell $t$ is in $\cA_\tau$ if and only if $r$
    and $s$ are both in $\cA_\tau$.

  \end{enumerate}
  We say that \emph{$j$ admits colax limits of arrows} if the colax
  limit exists for every arrow $f$ in $\cA_\la$.
\end{defn}

\begin{rmk}
  We can also consider the colax limit of a morphism in a mere
  2-category $\cA$, in which case $\cA = \cA_{\tau} = \cA_{\la}$ in
  the above definition and the third condition becomes vacuous.  We
  would then say that \emph{$\cA$ admits colax limits of arrows}.
\end{rmk}  

\begin{defn}[Lax doctrinal adjunction]
  We say that the pair $(H_\tau,H_\la)$ satisfies \emph{lax doctrinal
    adjunction} if given $f\cn A \to B$ in $A_\tau$ and an adjunction
  $(H_\tau f \dashv g, \epz, \eta)$ in $\cB$, there is a unique
  adjunction $(f \dashv \ol{g}, \ol{\epz}, \ol{\eta})$ in $A_\la$ such
  that $H_\la(f \dashv \ol{g}, \ol{\epz}, \ol{\eta}) = (H_\tau f
  \dashv {g}, {\epz}, {\eta})$.
\end{defn}

\begin{thm}[\cite{Bou2014Two}]\label{thm:2D-monadicity}
  Let $j\cn \cA_\tau \to \cA_\la$ be a 2-functor over $\cB$ as in
  \eqref{eq:2d-monadicity-setup}.
  Now suppose the following:
  \begin{enumerate}    

  \item\label{it:2-mon} $H_\tau$ is 2-monadic with associated 2-monad $T$;

  \item\label{it:j-colax-lim} $j$ admits colax limits of arrows in $\cA_\la$;

  \item\label{it:B-colax-lim} $\cB$ admits colax limits of arrows;

  \item\label{it:Hla} $H_\la$ is locally faithful and reflects identity 2-cells; and 
  
  \item\label{it:lax-doc-adj} $(H_\tau, H_\la)$ satisfies lax doctrinal adjunction.
  \end{enumerate}

  Then $(H_\tau, H_\la)$ is strict/lax monadic with associated 2-monad $T$.
\end{thm}

Our goal is often to identify the lax morphisms without explicitly
computing them.  \Cref{thm:2D-monadicity} accomplishes this by
identifying the lax morphisms as the 1-cells of $\cA_\la$, possibly up
to a 2-equivalence of 2-categories.

\subsection{Applications of 2-dimensional monadicity}
\label{sec:applications-of-Bourke-monadicity}

We now give the proofs of
\cref{prop:n-MonCat-lax-is-algebras,prop:diagrams,prop:red-diagrams}.

\begin{proof}[Proof of \cref{prop:n-MonCat-lax-is-algebras}]
  We apply \cref{thm:2D-monadicity}.  Let $\cA_\tau$ and $\cA_\la$,
  respectively, be the 2-categories of $n$-fold monoidal categories
  with strict, respectively lax, maps and $n$-fold monoidal transformations.
  Let $\cB = \Cat$, let $H_\tau$ and $H_\la$ be the respective
  forgetful functors and let $j$ be the inclusion.  The only
  conditions which are not immediate are that $(H_\tau, H_\la)$ satisfies lax
  doctrinal adjunction and that $j$ admits colax limits of lax arrows.
  However both are straightforward to verify, as we now
  sketch.

  To show that $(H_\tau, H_\la)$ satisfies lax doctrinal adjunction,
  suppose that $f\cn A \to B$ is a strict map of $n$-fold monoidal
  categories and that $g\cn B \to A$ is an adjoint to the underlying
  functor of categories.  Then one can construct a lax monoidal
  structure map for $g$ via the following composite
  \[
  g(b) \otimes_i g(b') \fto{\eta} 
  gf(g(b) \otimes_i g(b')) = 
  g(fg(b) \otimes_i fg(b')) \fto{g(\epz\otimes_i\epz)}
  g(b \otimes_i b').
  \]
  This is an instance of doctrinal adjunction for $\otimes_i$
  \cite{Kelly1974Doctrinal}.  One uses the strict structure of
  $f$ and the triangle identities to verify these lax monoidal
  structures are compatible with the interchange transformations.

  To show that $j$ admits colax limits of lax maps, one constructs the
  colax limit of underlying categories and verifies that it is endowed
  with an $n$-fold monoidal structure.  If $f\cn A \to B$ is a lax map
  of $n$-fold monoidal categories then the colax limit in $\Cat$ is a
  category $C$ whose objects are triples $(a, b, \si_{b,a})$ where $a
  \in A$, $b \in B$, and $\si_{b,a}\cn b \to f(a)$ is a morphism in
  $B$.  The morphisms of $C$ consist of pairs of morphisms between the
  component objects such that the obvious squares in $B$ commute.  For
  each index $i$, the $i$th monoidal product on $C$ is determined
  componentwise by the $i$th monoidal products on $A$ and $B$ and the
  lax monoidal structure maps for $f$.  The interchange maps are given
  pairwise by those in $A$ and $B$.  The compatibility of $f$ with
  interchange ensures that this defines a valid interchange for $C$.
  Verification of the necessary axioms consists of routine diagram
  algebra which we omit for the sake of brevity.

\end{proof}

\begin{proof}[Proofs of \cref{prop:diagrams,prop:red-diagrams}]
  We apply \cref{thm:2D-monadicity}.  Let $\cA_\tau$ be either $[I,
    \cK]_{\mathrm{red}}$ or $[I, \cK]$ and let $\cB = [\ob I, \cK]$
  and let $H_\tau$ be the map induced by the inclusion $\ob I \monoto
  I$.  Let $\cA_\la$ be the category of reduced or, respectively,
  unreduced diagrams with lax transformations, and let $j\cn \cA_\tau \monoto
  \cA_\la$ be the inclusion.

  We now verify the five conditions of \cref{thm:2D-monadicity}.
  Condition \eqref{it:2-mon} is proved in
  \cref{prop:i-monadic,prop:ij-monadic}.  Condition
  \eqref{it:B-colax-lim} follows because $I$ is small and $\cK$ is
  cocomplete. Condition \eqref{it:j-colax-lim} is straightforward by
  computing colax limits levelwise; i.e., in $\cB = [\ob I, \cK]$ and
  verifying that these extend to a 2-functor on
  $I$.
  The 2-dimensional aspect of
  the universal property for the levelwise colax limit ensures the 
  universal property of the colax limit in $\cA_\la$.

  We verify condition \eqref{it:Hla} in the unreduced case, noting
  that this immediately implies the same condition for the reduced
  case.  The functor $H_\la$ is the forgetful functor $[I, \cK]_l \to
  [\ob I, \cK]$.  This is clearly locally faithful and reflects
  identity 2-cells.

  For condition \eqref{it:lax-doc-adj}, let $f\cn X \to Y$ be a
  2-natural transformation of diagrams on $I$ and let $f_a \dashv g_a$
  be an adjunction for each $a \in \ob I$.  We construct a lax
  transformation $\ol{g}$ in the following way.  For $r\cn a \to b$ in
  $I$, define a 2-cell $\ol{g}_r\cn X(r)g_a \Rightarrow g_bY(r)$ as
  the composite
  \[
  X(r) g_a \xRightarrow{\eta_b 1 1} g_b f_b X(r) g_a = 
  g_b Y(r) f_a g_a \xRightarrow{1 1 \epz_a} g_b Y(r),
  \]
  where $\eta$, respectively $\epz$, are the unit, respectively
  counit, for the object-wise adjunction between $H_\la f$ and $g$.
  The middle equality is given by the strict naturality of $f$.  To
  see that $\ol{g}$ satisfies the axioms of a lax transformation one
  uses the triangle identities and 2-naturality of $f$.  Now $(f,
  \ol{g}, \epz, \eta)$ gives the unique adjunction lifting the
  object-wise adjunction and this completes the verification of
  \eqref{it:lax-doc-adj}.
\end{proof}

\bibliographystyle{amsalpha2}
\bibliography{lax_maps.bbl}%

\end{document}